\documentclass{amsart}
\usepackage[utf8]{inputenc}
\usepackage{lmodern}
\usepackage{amsthm}
\usepackage{mathtools}
\usepackage{hyperref}
\usepackage{amssymb}
\usepackage{amsmath}
\usepackage{csquotes}

\title{On polynomial extension property in n-disc}
\author{Krzysztof Maciaszek}
\date{March 7, 2019}

\begin{document}
\newtheorem{thrm}{Theorem}
\newtheorem*{thrmm}{Main result}
\newtheorem{df}[thrm]{Definition}
\newtheorem{prop}[thrm]{Proposition}
\newtheorem{lemma}[thrm]{Lemma}
\newtheorem{sublemma}[thrm]{Sublemma}
\newtheorem{cor}[thrm]{Corollary}
\newtheorem{obs}[thrm]{Observation}
\newtheorem{rem}[thrm]{Remark}

\maketitle
\let\thefootnote\relax\footnote{The author was supported by the NCN grant 2017/26/E/ST1/00723.}
\begin{abstract}
In this note we show that an one-dimensional algebraic subset $\mathcal{V}$ of arbitrarily dimensional polidisc $\mathbb{D}^n$, which has the polynomial extension property, is a holomorphic retract.
\end{abstract}

\section{The main result}

Let $f:\mathcal {V}\to \mathbb{C}$ be a function on some non-empty subset $\mathcal{V}$ of $\mathbb{C}^n$. We say that $f$ is holomorphic (on $\mathcal{V}$) if for every point $\zeta \in \mathcal{V}$ there exists $U$ an open neighbourhood of $\zeta$ in $\mathbb{C}^n$ and there exists a holomorphic function $F:U\to \mathbb{C}$ such that $F|_{U\cap \mathcal{V}}=f|_{U\cap \mathcal{V}}$.

The algebra of bounded holomorphic functions on $\mathcal{V}$ is denoted by $H^\infty (\mathcal{V})$ and we equip this space with the supremum norm:
$$||f||_{H^\infty(\mathcal{V})}=\sup_{\zeta \in \mathcal{V}}|f(\zeta)|.$$
By $\mathcal{P}(\mathcal{V})$ we denote the subalgebra of polynomials of $H^\infty(\mathcal{V})$.

\begin{df}
Take $U$ a bounded domain in $\mathbb{C}^n$ and let $\mathcal{V}$ be a non-empty subset of $U$. We say that $\mathcal{V}$ has the polynomial extension property with respect to $U$ if for every polynomial $f\in \mathcal{P}(\mathcal{V})$ there exists a bounded holomorphic function $F\in H^\infty (U)$ satisfying $F|_{\mathcal{V}}=f$ and
$$||F||_{H^\infty(U)}=||f||_{H^\infty (\mathcal{V})}.$$
\end{df}

Recall that $\mathcal{V}$ a subset of $U$ is called holomorphic retract if there exists a holomorphic map $r:U\to \mathcal{V}$ such that $r|_{\mathcal{V}}=\text{id}|_\mathcal{V}$.

We call the set $\mathcal{V}$ algebraic in $\mathbb{D}^n$ if there are polynomials $p_1,...,p_s$ such that $\mathcal{V}=\mathbb{D}^n\cap \bigcap_{j=1}^s p_j^{-1}(0)$.

The main result of the paper is:
\begin{thrm}\label{thrm1}
Let $n\geq 3$. Suppose that the set $\mathcal{V} \subset \mathbb{D}^n$ is an one-dimensional, algebraic and that it has the polynomial extension property. Then it is a holomorphic retract.
\end{thrm}
This theorem is an $n$-dimensional generalization of the result by Kosiński and M$^c$Carthy (cf. \cite{KosMc}, Theorem 5.1), where the authors proved it in case of the tridisc. In what follows, we will adopt their method, but with some modifications. Before we get to the proof, we will need some extra tools. For recent developments in the studies of extension properties and for some motivations we refer to \cite{KosMc2} and to the Agler, Lykova and Young paper \cite{Agl}.

\section{Preliminaries}
Recall the automorphism group of polidisc 
\begin{multline*}
    \text{Aut}(\mathbb{D}^n)=\lbrace \mathbb{D}^n\ni (z_1,...,z_n)\mapsto (e^{i\alpha_1}h_{w_1}(z_{\sigma (1)}),...,e^{i\alpha_n}h_{w_n}(z_{\sigma (n)}))\in \mathbb{D}^n:\\
    \alpha_j \in \mathbb{R}, w_j \in \mathbb{D}, j=1,...,n, \sigma \in \Sigma\rbrace,
\end{multline*}
where $\Sigma$ is the group of all permutations of $n$-elements and each $h_{w_j}$ is a M\"obius map.

We denote by $m$ the M\"obius distance on the unit disc
$$m(z,w):=\left | \frac{z-w}{1-z\overline{w}}\right |, \quad z,w\in \mathbb{D}.$$

Let $U$ be a domain in $\mathbb{C}^n$ and let $k_U$ be the Kobayashi pseudodistance for $U$:
$$k_U(z,w)=\inf \lbrace p(\lambda,\mu): \text{ there is } \varphi\in \mathcal{O}(\mathbb{D},U) \text{ such that } \varphi(\lambda)=z, \varphi(\mu)=w\rbrace,$$
where $p(\lambda,\mu)=\text{tanh}^{-1}(m(\lambda,\mu))$ is the Poincar\'e distance on the unit disc $\mathbb{D}$ in $\mathbb{C}$. Denote by $c_U$ the Carath\'eodory pseudodistance $$c_U(z,w)=\sup \lbrace p(f(z),f(w)):f\in \mathcal{O}(U,\mathbb{D})\rbrace, \quad z,w\in U.$$

We say that the function $f \in \mathcal{O}(\mathbb{D},U)$ is a Kobayashi extremal map with respect to two distinct points $z,w\in U$ if there are $\lambda, \mu \in \mathbb{D}$ such that $f(\lambda)=z, f(\mu)=w$ and $$k_U(z,w)=p(\lambda, \mu).$$
Similarly we define the Carath\'eodory extremal map for a pair of points $z$ and $w$, as a map $g\in \mathcal{O}(U,\mathbb{D})$ which satisfies
$$c_U(z,w)=p(g(z),g(w)).$$

Finally, a complex geodesic is a range of Kobayashi extremal map. In case when the bounded domain $U$ is convex, it follows from the Lempert theorem (see e.g. \cite{Jar}, Theorem 11.2.1) that for any pair of points in $\mathbb{D}$ and for every Kobayashi extremal $f$ through these points there exists a Carath\'eodory extremal map $g$, being a left-inverse to $f$, which means that $g\circ f=\text{id}|_{\mathbb{D}}$.

\begin{df}
Take two points $z,w$ in polydisc $\mathbb{D}^n$ such that $z\neq w$. We say that $z$ and $w$ form a $k$-balanced pair, where $2\leq k \leq n$, if after a possible permutation of coordinates, we have
$$m (z_1,w_1)=...=m (z_k,w_k)\geq m(z_{j},w_{j}), \quad j\in \lbrace k+1,...,n\rbrace.$$
If $w=0$ then we will say that $z$ is a $k$-balanced point.
\end{df}

Recall that the subset $V$ of a domain $U$ is called \textit{relatively polynomially convex} if $\overline{V}$ is polynomially convex and $\overline{V}\cap U=V$. 

\begin{lemma}[cf. \cite{KosMc}, Lemma 2.2]\label{connectedV}
If $U$ is a bounded domain and $\mathcal{V}$ is a relatively polynomially convex subset of $U$ that has the polynomial extension property, then $\overline{\mathcal{V}}$ is connected.
\end{lemma}

\begin{thrm}[cf. \cite{KosMc}, Theorem 2.3]\label{carpick}
Let $U$ be bounded, and assume that $\mathcal{V}\subset U$ has the polynomial extension property. Let $\varphi$ be a Carath\'eodory-Pick extremal for $U$ for some data. If $\varphi|_{\mathcal{V}}$ is in $\mathcal{P}(\overline{\mathcal{V}})$, then $\overline{\varphi (\mathcal{V})}$ contains the unit circle $\mathbb{T}$.
\end{thrm}

\begin{lemma}\label{balanced}
Let $\mathcal{V}\subset \mathbb{D}^n$ be a relatively polynomially convex set that has the polynomial extension property. Suppose that $\mathcal{V}$ contains an $n$-balanced pair $(a,b)$. Then $\mathcal{V}$ contains an $n$-balanced disc of the form
$$\lbrace (h_{w_1}(\lambda),...,h_{w_n}(\lambda)):\lambda \in \mathbb{D}\rbrace,$$
where $h_{w_j}$ are M\"obius maps, $j=1,...,n$.
\end{lemma}
\begin{proof}
Take an automorphism that interchanges $b$ with $0$ and $a$ with some point $\alpha \in \mathbb{D}^n\setminus \lbrace 0 \rbrace$ such that $\alpha_1=...=\alpha_n$. Set $\varphi(z)=(z_1+...+z_n)/n$. Since $c_{\mathbb{D}^n}(z,w)=\max \lbrace m(z_j,w_j): j=1,...,n\rbrace$ (cf. Corollary 2.3.5. in \cite{Jar}), it follows that $\varphi$ is a Carath\'eodory extremal for $(a,b)$. By Theorem \ref{carpick}, the set $\overline{\mathcal{V}}$ contains the unit circle $\lbrace (\zeta,...,\zeta):\zeta \in \mathbb{T}\rbrace$, and the polynomial convexity of $\overline{\mathcal{V}}$ leads to the conclusion that $\mathcal{V}$ contains $n$-balanced disc $\lbrace (\lambda,...,\lambda): \lambda \in \mathbb{D}\rbrace$.
\end{proof}

Let $\pi_k:\mathbb{C}^n\to \mathbb{C}^k$ denote the natural projection on the first $k$-coordinates $z_1,...,z_k$.
\begin{lemma}\label{cosik}
If we take a point $a=(a',a'')\in \mathbb{C}^k\times \mathbb{C}^{n-k}$ which is an isolated point of $\pi_k^{-1}(a')\cap \mathcal{V}$, then we are able to find a polydisc $U=U'\times U'' \subset \mathbb{C}^k\times \mathbb{C}^{n-k}$ centered at $a$ and such that $\pi_k:U\cap \mathcal{V}\to U'$ is proper.
\end{lemma}

\begin{prop}[\cite{Chi}, Theorem 3.7]\label{proper1}
Let $\mathcal{X}$ be an analytic set in $\mathbb{C}^n$ and $a\in \mathcal{X}$ with $\text{dim}_a\mathcal{X}=k$. If there is a connected neighborhood $a\in U=U'\times U''$ such that $\pi_k:U\cap \mathcal{X}\to U'\subset \mathbb{C}^k$ is proper, then there exists an analytic set $\mathcal{Y}\subset U'$, $\text{dim}\mathcal{Y}<k$ and $p\in \mathbb{N}$ such that
\begin{enumerate}
    \item $\pi_k:U\cap \mathcal{X}\setminus \pi_k^{-1}(\mathcal{Y})\to U'\setminus \mathcal{Y}$ is a locally biholomorphic $p$-sheeted cover. In particular, $\# \pi_k^{-1}(z')\cap \mathcal{V}\cap U=p$ for all $z'\in U'\setminus \mathcal{Y}$.
    \item $\pi_k^{-1}(\mathcal{Y})$ is nowhere dense in $\mathcal{X}_{(k)}\cap U$, where $\mathcal{X}_{(k)}=\lbrace z\in \mathcal{X}:\text{dim}_z \mathcal{X}=k\rbrace$.
\end{enumerate}
\end{prop}

\begin{prop}[\cite{Chi}, Proposition 3.3.3.]\label{loj}
Let $G'\subset \mathbb{C}^k, G''\subset \mathbb{C}^m$ be open subsets. Suppose $V$ is analytic in $G=G'\times G''$, and take a projection $\pi_k:(z',z")\mapsto z'$. Assume that $V'=\pi_k(V)$ is a complex submanifold in $G'$ and that $\pi_k:V \to V'$ is one-to-one. Then $V$ is a complex submanifold in $G$ and $\pi_k:V\to V'$ is a biholomorphic map. 
\end{prop}

The next lemma is standard.
\begin{lemma}\label{identity}
Let $\mathcal{X},\mathcal{Y}$ be analytic where $\mathcal{X}$ is additionally irreducible. Then if $w\in \mathcal{X}\cap \mathcal{Y}$ and the germs $(\mathcal{X})_w$ and $(\mathcal{Y})_w$ are equal, then $\mathcal{X}\subset \mathcal{Y}$.
\end{lemma}

To obtain the main result, we will prove that $\mathcal{V}$ is a graph of holomorphic map. Then we will apply the following theorem by Heath and Suffridge \cite{HeSu}: 
\begin{thrm}\label{thrmk}
The set $\mathcal{V}$ is a holomorphic retract of $\mathbb{D}^n$ if and only if, after a permutation of coordinates, $\mathcal{V}$ is a graph of a map from $\mathbb{D}^p$ to $\mathbb{D}^{n-p}$ for some $0\leq p \leq n$.
\end{thrm}

\section{Proof of the main result}
For any regular point $w$ there exists a unique irreducible component containing it. Recall that the properties of having the polynomial extension property, connectedness, being relatively polynomially convex and being a retract are all invariant under application of an automorphism of $\mathbb{D}^n$. Hence, for given $k$-balanced pair of points $(z,w)$ which is not $(k+1)$-balanced we can assume that $z=0$ and $w=(w_1,...,w_1,w_{k+1},...,w_n)$ with $|w_1|>|w_j|$ for $j>k$.

\begin{lemma}\label{lem2}
Assume that $\mathcal{V}$ is an algebraic and one-dimensional set in $\mathbb{D}^n$ that satisfies the polynomial extension property. Suppose that $\mathcal{V}$ contains both $0$ and $w=(w_1,w_1,...,w_1,w_{k+1},...,w_n)$, where $w$ is a regular point of $\mathcal{V}$ with $|w_1|>|w_j|$ for $j=k+1,...,n$ and $2\leq k \leq n-1$. Then the irreducible component $S\subset \mathcal{V}$ passing through $w$ is contained in $\lbrace z_1=z_2=...=z_k \rbrace$.
\end{lemma}

\begin{proof} Suppose that $S$ is not contained in $\lbrace z_1=...=z_k \rbrace$.

Since $w$ is a regular point, by Lemma \ref{connectedV} it is not isolated in $\mathcal{V}$. Hence locally around $w$ the set $S$ can be described by holomorphic functions
$$\lbrace (f_1(\lambda),...,f_n(\lambda)):\lambda \in \mathbb{D}(\epsilon)\rbrace, \quad \text{where }\epsilon>0.$$
Here we can assume that $f_1(0)=...=f_k(0)=w_1$ and $f_p(0)=w_p$ whenever $p\in \lbrace k+1,...,n \rbrace$. Decreasing $\epsilon$ we may also assume that the above set is contained in $\lbrace |z_{k+1}|,...,|z_n|<|z_1|,...,|z_k|\rbrace$. Define \begin{equation*}
    u(\lambda)=\max_{1\leq j \leq k}\lbrace |f_{j}(\lambda)|\rbrace, \quad \text{ for }\lambda \in \mathbb{D}(\epsilon),
\end{equation*} and let $$S_j=\lbrace \lambda \in \mathbb{D}(\epsilon): u(\lambda)=|f_j(\lambda)| \rbrace.$$ We claim that it is possible to find a pair of indices $p,q\in \lbrace 1,...,k\rbrace$ with $p\neq q$ and a sequence $\lbrace \lambda_n \rbrace_{n=1}^\infty \subset S_p\cap S_q$ converging to zero.

Suppose the contrary. Since the functions describing first $k$ coordinates coincide at a point $0$, then for all pairs of indices $1\leq i,j \leq k$ we have $|f_i|=|f_j|$ on some subset $\Gamma_{ij}\subset \mathbb{D}(
\epsilon)$ with $0\in \Gamma_{ij}$ being an accumulation point. If there exists an index $p$ such that $|f_p|\geq|f_j|$ on $\mathbb{D}(\epsilon)$ for all $j\neq p$, then we are done. Otherwise for some $p\neq q$ we can pick two sequences $(\lambda_n)_{n=1}^\infty$ in $S_{p}$ and $(\mu_n)_{n=1}^\infty$ in $S_{q}$, both of them converging to $0\in \mathbb{D}(\epsilon)$. Let permute the indices such that $p=1, q=2$.

After possibly passing to a subsequence, any pair of points $\lambda_n$ and $\mu_n$ we can connect by an arc in $\mathbb{D}(\epsilon)$ in a way that any $n$-th arc does not meet the previous $n-1$ arcs. On every such $n$-th curve, by continuity argument, we can find a point $\nu_n$ such that for some $p\in \lbrace 2,...,k \rbrace$ we have $|f_1(\nu_n)|=|f_p(\nu_n)|\geq |f_j(\nu_n)|$ for all $j\neq 1,p$. Since there is only a finite number of coordinates, we can again pass to a subsequence if necessary, and find an index $i_0$ such that $|f_1(\nu_n)|=|f_{i_0}(\nu_n)|$ for all $(\nu_n)_{n=1}^\infty$. This proves the claim.

Now we have two possibilities.

\textbf{Case one:} There is an index $j_0$ such that we have $f_1(\nu_n)=\omega_n f_{j_0}(\nu_n)$ for $E:=\lbrace \omega_n:n\in \mathbb{N}\rbrace$ being an infinite subset of $\mathbb{T}$ and such that $1\in \mathbb{T}$ is an accumulation point of $E$. After a possible permutation of coordinates we may assume that $j_0=2$.

Consider the mappings
\begin{align*}
    F_{\omega_n}:&\mathbb{D}^n\to \mathbb{D}
    \\ &z\mapsto \frac{z_1+\overline{\omega_n}z_2}{2}.
\end{align*}
For any $\omega_n \in E$ we can take a point $a=(a_1,\omega_n a_1, a_3,...,a_n):=(f_1(\nu_n),...,f_n(\nu_n))$ with $|a_1|\geq |a_j|$ for all $j=3,...,n$ and there is a geodesic
$$\lambda \mapsto \left(\lambda, \omega_n \lambda, \frac{a_3}{a_1}\lambda,...,\frac{a_n}{a_1}\lambda \right)$$
in $\mathbb{D}^n$ passing through $0$ and $a$. Now, $F_\omega$ is a left inverse to the Kobayashi extremal through these points and by Theorem \ref{carpick} we obtain that $\mathbb{T}\subset \overline{F_{\omega_n} (\mathcal{V})}$. Therefore $\mathbb{T}\times E \subset \pi_{2} (\overline{\mathcal{V}})$. On the other hand, the variety $\mathcal{V}$ is one-dimensional and algebraic, so over any point $\zeta$ in $\mathbb{T}$ (except perhaps some zero-dimensional singular set) there is only a finite number of points lying over $\zeta$. Hence we get a contradiction and so $S \subset \lbrace z_1=...=z_k\rbrace$.

\textbf{Case two:} There is an index $j_0$ and a sequence $(\nu_n)_{n=1}^\infty$ such that $f_1=f_{j_0}$ on that sequence and hence on the whole $\mathbb{D}(\epsilon)$. Then we have $S\cap \mathbb{B}_w(t_0) \subset \lbrace z_1=z_{j_0} \rbrace$ for a suitable $t_0>0$. Here $\mathbb{B}_w(t_0)$ is a ball centered at $w$. Since $S$ is irreducible, we have from Lemma \ref{identity} that in fact $S\subset \lbrace z_1=z_{j_0}\rbrace$. 

Choose all $j\in \lbrace 2,...,n\rbrace$ for which we can find such a sequence $(\nu_n)_{n=1}^\infty$ where $f_1(\nu_n)=f_j(\nu_n)$ for all $n$ and repeat the above argument. After a possible permutation of indices we get $S\subset \lbrace z_1=...=z_l\rbrace$. Now, if $l=k$ then we have finished. However, if $l<k$ then for some $j_0\in \lbrace l+1,...,k\rbrace$ we would find ourselves in a situation from Case one, which leads to the contradiction.
\end{proof}

\begin{lemma}\label{lem0}
If $\mathcal{V}$ is an algebraic subset of $\mathbb{D}^n$ of dimension one that satisfies the polynomial extension property, then locally $\mathcal{V}$ is a graph of a holomorphic function.
\end{lemma}
The local character of the Lemma \ref{lem0} allows us to make some simplifications. It is enough to prove that $\mathcal{V}$ is smooth at $0\in \mathcal{V}$. Moreover, any analytic set can be decomposed to its irreducible components and locally we can choose finite number of them. Therefore we can choose $t>0$ for which all points in $\mathcal{V}\cap \mathbb{B}(t)\setminus \lbrace 0 \rbrace$ are regular and we can write $\mathcal{V}\cap \mathbb{B}(t)=U_1\cup ... \cup U_s$ for some $s\in \mathbb{N}$, where $U_j$ are irreducible components passing through zero (and hence $U_i\cap U_j=\lbrace 0 \rbrace$ whenever $i\neq j$).

We will prove the Lemma \ref{lem0} in a batch of sublemmas. 

\begin{sublemma}\label{lem3}
If there exists a point $w\in \mathcal{V}\cap \mathbb{B}(t)$ such that $|w_1|>|w_j|$ for $j=2,...,n$, then the irreducible component of $\mathcal{V}\cap \mathbb{B}(t)$ passing through $0$ and $w$ is a graph $\lbrace (\lambda, \lambda f_2(\lambda),...,\lambda f_n(\lambda)):\lambda \in \mathbb{D}(\epsilon)\rbrace$, where $f_j$ are in the open unit ball of $H^\infty (\mathbb{D}(\epsilon))$.
\end{sublemma}
\begin{proof} 

Take the projection to the first variable \begin{align*}
    \pi_1:\mathbb{D}^n&\to \mathbb{D},\\
    z&\mapsto z_1.
\end{align*}  
Let $S$ be the irreducible component of $\mathcal{V}$ passing through $0$ and $w$. Note, that $\pi_1^{-1}(0)\cap S=\lbrace 0 \rbrace$. Otherwise we could find point $x=(0,x_2,...,x_n)\in S$ with at least one $x_j\neq 0$. Since $S$ is irreducible, we could connect $x$ and $w$ by a curve on which there would be a point that forms with zero a pair which is at least $2$-balanced. Then Lemma \ref{lem2} would give a contradiction.

Now, it follows from Lemma \ref{cosik} and Proposition \ref{proper1} that there exists a polydisc $U$ such that $\pi_1|_{U\cap S}$ is an analytic $d$-sheeted covering with at most $0$ being the critical point. If $d=1$ then by the analytic graph theorem the irreducible component $S$ is a graph of holomorphic function. 

Suppose for the sake of contradiction that $d>1$. We shall note that there is no nonzero point $x \in S\cap \lbrace |z_1|=...=|z_n| \rbrace$. Indeed, if there is one, then we can choose such a point $x$ and a curve $\gamma:[0,1]\to S$ connecting $w=\gamma(0)$ and $x=\gamma(1)$, and such that $\gamma (t) \notin S\cap \lbrace |z_1|=...=|z_n| \rbrace$ for any $t\in [0,1)$. Then we apply Lemma \ref{balanced} and we get $n$-balanced disc through $x$ additionally to the curve $\gamma$, which contradicts the assumption that $x$ is a regular point. 

Therefore we can assume that there is no point in $S\setminus \lbrace 0 \rbrace$ with all the coordinates equal on modulus. The Lemma \ref{lem2} allows us additionally to note that there is also no $k$-balanced point for $2\leq k\leq n-1$. So we can see that $|z_1|>|z_j|$, $j=2,...,n$ for all nonzero $z$ in $S$. Otherwise, by a continuity argument we would find a nonzero point $x\in S$ such that $|x_1|=|x_j|$ for some $j\in \lbrace 2,...,n\rbrace$ and this as we noted is not possible.

Since we assumed that $d>1$, we can find another point $w'$ different than $w$ in the fiber $\pi_1^{-1}(w_1)\cap S$, with some coordinates $w'=(w_1,w_2',...,w_n')$. Let $\varphi$ be the automorphism which maps
$$z\mapsto (h_{w_1}(z_1),...,h_{w_n}(z_n)),$$
where $h_{w_j}$ is a M\"obius map. Then
\begin{align*}
    w &\text{ is interchanged with } 0;\\
    w' &\text{ is mapped to } w'':=(0,h_{w_2}(w_2'),...,h_{w_n}(w_n')).
\end{align*}

Now, all points except at most $w$ are regular in $\varphi (S)$. Let $\gamma:[0,1]\to \varphi (S)$ be a curve that joins $\gamma(0)=w$ and $\gamma(1)=w''$. We can choose $\gamma$ such that it does not pass through zero. It is easy to see that there has to be a point $t_0\in (0,1)$ such that $\gamma (t_0)$ is $k$-balanced for some $k\geq 2$. Again Lemma \ref{balanced} and Lemma \ref{lem2} give a contradiction.

Hence near zero $S$ is a graph of a holomorphic function. Additionally we have shown that there is no $k$-balanced point for any $k\geq 2$, which means that $|f_j|<1$ on $\mathbb{D}(\epsilon)$ for all $j=2,...,n$. 
\end{proof}

\begin{sublemma}\label{lem4}
Fix $k \in \lbrace 2,..., n-1\rbrace$. Suppose there is a point $w\in \mathcal{V}\cap \mathbb{B}(t)$ such that $|w_1|=...=|w_k|>|w_j|$ for $j\in \lbrace k+1,...,n\rbrace$. Then the irreducible component passing through $0$ and $w$ is a graph $\lbrace (\lambda,\omega_2 \lambda,...,\omega_k \lambda, \lambda f_{k+1}(\lambda),...,\lambda f_n(\lambda)):\lambda \in \mathbb{D}(\epsilon)\rbrace$, where $f_j\in H^\infty (\mathbb{D}(\epsilon))$ and $\omega_j\in \mathbb{T}$.
\end{sublemma}

\begin{proof}
The proof is essentially the same as the proof of Sublemma \ref{lem3}, so we just make few additional observations.

It follows from Lemma \ref{lem2} that in the component $S$ containing $w$ there is no $s$-balanced points with $k<s\leq n-1$. 

Take an automorphism which multiplies every of the first $k$ coordinates by suitable unimodular constants, so we are allowed to assume that $w_1=...=w_k$. Therefore the unique component $S$ passing through $0$ and $w$ is a subset of $\lbrace z_1=...=z_k \rbrace$. Take the projection $\pi_k$ to the first $k$ variables, which again we can assume to be proper when restricted to $S\cap U$ for $U$ being some polydisc.

We need only to show, that the projection is single sheeted. Suppose then that there is another point $w'=(w_1,...,w_1,w_{k+1}',...,w_n')\in \pi_k^{-1}((w_1,...,w_1))\cap S \cap U$, different than $w$. Since $S$ is contained in $\lbrace z_1=...=z_k\rbrace$, we have already reduced the problem to the situation from the Sublemma \ref{lem3}, so we can apply it to get the result.
\end{proof}

\begin{sublemma}\label{lem5}
Suppose $\mathcal{V}$ contains two discs $\mathcal{D}_1:=\lbrace (\lambda f_1(\lambda),\lambda,\lambda f_3(\lambda),...,\lambda f_n(\lambda)): \lambda \in \mathbb{D}(\epsilon)\rbrace$ and $\mathcal{D}_2:=\lbrace (\lambda, \lambda g_2(\lambda),...,\lambda g_n(\lambda)): \lambda \in \mathbb{D}(\epsilon)\rbrace$, where all $f_j,g_j$ are in the closed unit ball in $H^\infty({\mathbb{D}}(\epsilon))$. Suppose there is no open set $U$ such that $\mathcal{D}_1\cap U=\mathcal{D}_2\cap U$. Then in arbitrarily small neighbourhood of the origin we can find two points $w\in \mathcal{D}_1\setminus \lbrace 0 \rbrace$ and $w'\in \mathcal{D}_2\setminus \lbrace 0 \rbrace$ such that $(w,w')$ forms a pair, which is at least $2$-balanced.
\end{sublemma}

\begin{proof}
If both $f_1,g_2\equiv 0$, the result is obvious. We can assume at least one of them, say $f_1$ is not the zero function. Take some point $w=(w_1,...,w_n)\in \mathcal{D}_1$ with $w_2=\lambda_0$ and $w_j=\lambda_0 f_j(\lambda_0)$ for $j\in \lbrace 1,...,n\rbrace \setminus \lbrace 2 \rbrace$. Choose $\lambda_0$ so that $f_1(\lambda_0)\neq 0$. Then we pick another point $x=(x_1,...,x_n):=(\lambda_1,\lambda_1 g_2(\lambda_1),...,\lambda_1 g_n(\lambda_1))$ in $\mathcal{D}_2$ so it satisfies $\lambda_1=\lambda_0 f_1(\lambda_0)$. 

Take an automorphism $\varphi$ of the unit polydisc that interchanges $0$ and $x$:
$$z \overset{\varphi}{\mapsto} \left ( h_{x_1}(z_1),...,h_{x_n}(z_n) \right ).$$
Since any point in $\mathcal{D}_1$ can be joined with $0$ by a curve, we can also link $x$ with every point from $\varphi(\mathcal{D}_1)$. Take an appropriate curve $\gamma: [0,1]\to \varphi (\mathcal{V})$ with $\gamma(0)=\varphi (w)=(0,m_{x_2}(w_2),...,m_{x_n}(w_n))$ and $\gamma(1)=\varphi (0)=x$. By the continuity argument the curve $\gamma$ has to meet a point which is at least $2$-balanced.
\end{proof}

\begin{proof}[Proof of the Lemma \ref{lem0}]
The Lemma \ref{balanced}, Sublemma \ref{lem3} and Sublemma \ref{lem4} show that any irreducible component passing through $0$ in its neighbourhood is a graph of holomorphic function. Suppose that the variety $\mathcal{V}$ is a union of at least two graphs. After a permutation of coordinates $\mathcal{V}$ contains two discs:
\begin{align}\label{disc1}
\mathcal{D}_1&=\lbrace (\lambda,\lambda f_{2}(\lambda),...,\lambda f_n(\lambda)): \lambda \in \mathbb{D}(\epsilon)\rbrace;\\ 
\label{disc2}
\mathcal{D}_2&=\lbrace (\lambda g_1(\lambda),\lambda,\lambda g_3(\lambda),...,\lambda g_{n}(\lambda)): \lambda \in \mathbb{D}(\epsilon)\rbrace,
\end{align}
for $f_j, g_j$ being in closed unit ball of $H^\infty (\mathbb{D}(\epsilon))$ and with $0$ being the only point of intersection of these two discs.

Observe that none of points from $\mathcal{D}_1\setminus \lbrace 0 \rbrace$ forms $n$-balanced pair with any point of $\mathcal{D}_2 \setminus \lbrace 0 \rbrace$. Otherwise by Lemma \ref{balanced} there is $n$-balanced disc through these points contradicting the assumption, that the only intersection point of $\mathcal{D}_1$ and $\mathcal{D}_2$ is the zero point.

Suppose there is $k$-balanced pair $(x,y)\in \mathcal{D}_1\times \mathcal{D}_2$ for $2\leq k \leq n-1$. Choose $k$ that is the largest possible. Then by Lemma \ref{lem2}, after perhaps permutation of coordinates we have 
\begin{align}\label{inclusionhehe}
    \mathcal{D}_1\setminus \lbrace 0 \rbrace \subset \lbrace z:m(z_1,y_1)=...=m(z_k,y_k)>m(z_j,y_j)\text{ for }j=k+1,...,n\rbrace,\\
    \label{inclusionhehe2}
    \mathcal{D}_2 \setminus \lbrace 0 \rbrace \subset \lbrace w:m(x_1,w_1)=...=m(x_k,w_k)>m(x_j,w_j)\text{ for }j=k+1,...,n\rbrace.
\end{align}
This implies, that in fact any pair $(z,w)\in \mathcal{D}_1\times \mathcal{D}_2$ except $(0,0)$ is $k$-balanced. In particular if we take $z=0\in \mathcal{D}_1$, we see that $\lambda g_1(\lambda)\neq 0$ for all $\lambda \in \mathbb{D}(\epsilon)\setminus \lbrace 0 \rbrace$. Now, for some $\lambda_0\in \mathbb{D}(\epsilon)\setminus \lbrace 0 \rbrace$ choose $\lambda_1 \in \mathbb{D}(\epsilon)$ such that $\lambda_1=\lambda_0g_1(\lambda_0)$. Therefore $m(\lambda_1,\lambda_0g_1(\lambda_0))=0$ contradicting inclusions (\ref{inclusionhehe}) and (\ref{inclusionhehe2}). Hence there is no $k$-balanced point either.

Finally, recall that discs (\ref{disc1}) and (\ref{disc2}) intersect only at point zero. So, by Sublemma \ref{lem5} there is always a pair of point which is at least $2$-balanced, but this is impossible by what we just showed.

This finishes the proof.
\end{proof}

\begin{proof}[Proof of the Theorem \ref{thrm1}]We know already that $\mathcal{V}$ is smooth, i.e. free from singular points. Acting on $\mathcal{V}$ with an automorphism we lose no generality if we assume that $0\in \mathcal{V}$.

Suppose that $\mathcal{V}$ contains a $k$-balanced pair of points for $k\geq 2$ and take such a pair where $k$ is the largest possible. If $k=n$ then $\mathcal{V}$ contains a $n$-balanced disc $\mathcal{D}$ due to Lemma \ref{balanced} and that disc covers the whole $\mathcal{V}$, so Theorem \ref{thrmk} finishes the proof. Indeed, we can assume up to automorphism that $\mathcal{D}=\lbrace (\lambda,...,\lambda):\lambda \in \mathbb{D}\rbrace$. If there exists some point $x \in \mathcal{V}\setminus \mathcal{D}$, then for any point $w=(\lambda_0,...,\lambda_0)\in \mathcal{D}$, the pair $(x,w)$ is at most $k$-balanced for some $k\leq n-1$. By Lemma \ref{lem2} and after permutation of coordinates we have $m(x_1,\lambda)=...=m(x_k,\lambda)>m(x_j,\lambda)$ for $j=k+1,...,n$, for all $\lambda \in \mathbb{D}$. Taking $\lambda=x_1$ we get a contradiction.

Suppose now $k<n$. It follows again from Lemma \ref{lem2} that there is no $j$-balanced pair in $\mathcal{V}$ which is not $k$-balanced, for any $2\leq j<k$. Hence we can assume that
\begin{equation}\label{last}
\mathcal{V}\subset \lbrace z\in \mathbb{D}^n:|z_1|=...=|z_k|>|z_j|\ \text{for all } j=k+1,...,n\rbrace \cup \lbrace 0 \rbrace.    
\end{equation}

Take the projection to the first coordinate $\pi_1: \mathcal{V}\to \mathbb{D}$ which by virtue of the inclusion (\ref{last}) is proper. Recall that for any $a\in \mathcal{V}$ we have $\text{dim}_a\mathcal{V}=1$. By Proposition \ref{proper1} the projection $\pi_1$ is a local $p$-sheeted covering except at most discrete set of singular points, for some $p\in \mathbb{N}$. But the set of singularities is empty. Suppose now that $p>1$. Then since we assumed that zero is in $\mathcal{V}$, we can now pick two points $\lbrace 0,w\rbrace \in \pi_1^{-1}(0)\cap \mathcal{V}$ where $w=(0,w_2,...,w_n)\neq 0$. But this contradicts the assumption that for $\mathcal{V}$ the inclusion (\ref{last}) is satisfied, so in fact $p=1$.

Therefore $\pi_1:\mathcal{V}\to \mathbb{D}$ is an one-to-one map, so by Proposition \ref{loj} it is a biholomorphism. Now, Theorem \ref{thrmk} finishes the proof.
\end{proof}

\textbf{Remarks.} If we strengthen the assumptions in Theorem 2, demanding that the variety $\mathcal{V}$ has the $H^{\infty}$-extension property instead of the polynomial one, then $\mathcal{V}$ is connected. Bearing in mind first three paragraphs of the \textit{Proof of the Theorem 2} we observe that 
$$\mathcal{V}\subset \lbrace{z\in \mathbb{D}^n:|z_1|>|z_j|, \text{ for all } j=2,...,n}\rbrace\cup \lbrace 0 \rbrace.$$

Define $f:=z_n/z_1\in H^\infty(\mathcal{V})$. There exists $F\in H^\infty (\mathbb{D}^n)$ an extension of $f$. Then $g:=z_1F$ is an nontrivial extension of $z_n|_{\mathcal{V}}$ so there are satisfied assumptions of the following theorem by Guo, Huang and Wang \cite{Guo}:
\begin{thrm}
Suppose $\mathcal{V}$ is a subset of $\mathbb{D}^n$ and $z_n|_{\mathcal{V}}$ has a nontrivial extension. Then $\mathcal{V}$ has the form $\mathcal{V}=\lbrace (z',f(z')):z'\in \pi_{n-1}(\mathcal{V}))\rbrace$, where $f$ is the unit ball of $H^\infty(\mathbb{D}^{n-1})$.
\end{thrm}

\textbf{Acknowledgments.} I would like to express my gratitude to Łukasz Kosiński for his guidance and numerous suggestions.

 \textsc{\\ \\ \indent Institute of Mathematics, Jagiellonian University, Łojasiewicza 6, 30-348
Kraków, Poland. }

\textit{E-mail address:}  \href{mailto:krzysztof.maciaszek@im.uj.edu.pl}{krzysztof.maciaszek@im.uj.edu.pl} 
\end{document}